\documentclass{article}
\usepackage{amsmath, amsthm, amssymb,tikz,hyperref}
\usepackage{graphicx}
\usepackage{color}
\usepackage{verbatim}
\usepackage{subfigure}

\newcommand{\cover}{\lessdot}
\newcommand{\ree}[1]{(\ref{#1})}

\newtheorem{thm}{Theorem}
\newtheorem{lem}[thm]{Lemma}
\newtheorem{cor}[thm]{Corollary}

\newtheorem{prop}[thm]{Proposition}
\theoremstyle{definition}
\newtheorem{defi}[thm]{Definition}
\newtheorem{example}[thm]{Example}
\DeclareMathOperator{\supp}{supp}

\newcommand{\bfm}[1]{\boldsymbol{#1}}

\usepackage[latin1]{inputenc}
\usepackage{subfigure}

\begin{document}
\title{Factoring the characteristic polynomial of a lattice}
\author{Joshua Hallam\\ Bruce E. Sagan\\Department of Mathematics\\Michigan State University
\\ East Lansing, MI 48824-1027, USA}
\date{\today\\[10pt]
	\begin{flushleft}
	\small Key Words: characterstic polynomial, factorization, increasing forest, lattice, M\"obius function, quotient, semimodular, supersolvable
	                                       \\[5pt]
	\small AMS subject classification (2010):  06A07 (Primary), 05A15 (Secondary)
	\end{flushleft}}

\maketitle

\begin{abstract}
We introduce a new method for showing that the roots of the characteristic polynomial of  certain finite lattices are all nonnegative integers. This method is based on the  notion of a quotient of a  poset which will be developed to explain this factorization.  Our main theorem will give two simple conditions under which the characteristic polynomial factors with nonnegative integer roots. We will see that Stanley's Supersolvability Theorem is a corollary of this result. Additionally, we will  prove a theorem which gives three conditions equivalent to factorization. To our knowledge, all other theorems in this area only give conditions which imply factorization. This theorem  will be used to  connect the generating function for increasing spanning forests of a graph to its chromatic polynomial. We finish by mentioning some other applications of quotients of posets as well as some open questions.

\end{abstract}

\section{Introduction}

For the entirety of this paper let us assume that all our partially ordered sets (posets)  are finite, ranked, and contain a unique minimal element, denoted $\hat{0}$. Recall the one-variable M\"{o}bius function of a poset, $\mu: P \rightarrow \mathbb{Z} $, is defined recursively by 
$$
\sum_{y\leq x} \mu(y) =\delta_{\hat{0},x}
$$
where $\delta_{\hat{0},x}$is the Kronecker delta. 

Also, recall that a poset $P$, is \emph{ranked} if, for each $x\in P$, every saturated $\hat{0}$--$x$ chain has the same length. Given a ranked poset, we get a \emph{rank function} $\rho:P \rightarrow \mathbb{N} $ defined by setting $\rho(x)$ to be the length of a $\hat{0}$--$x$ saturated chain. 
We define the \emph{rank} of a ranked poset $P$ to be
$$
\rho(P) = \max_{x\in P} \rho(x).
$$
When $P$ is ranked, the generating function for $\mu$ is called the \emph{characteristic polynomial} and is given by 
$$
\chi(P,t) = \sum_{x\in P} \mu(x) t^{\rho(P)- \rho(x)}.
$$
We are interested in identifying lattices which have characteristic polynomials with only nonnegative integer roots. In this case, we also wish to show that the roots are the cardinalities of sets of atoms of the lattice.

Before we continue, let us mention some previous work done by others on the factorization of the characteristic polynomial. For a more complete overview, we suggest reading the survey paper by Sagan~\cite{s:wcpf}. In~\cite{s:sl}, Stanley showed that the characteristic polynomial of a semimodular supersolvable lattice always has nonnegative integer roots. Additionally, he showed these roots were given by the sizes of blocks in a partition of the atom set of the lattice. Blass and Sagan~\cite{bs:mfl} extended this result to LL lattices. In~\cite{z:sgc}, Zaslavsky generalized the concept of coloring of graphs to coloring of signed graphs and showed how these colorings were related to the characteristic polynomial of certain hyperplane arrangements. This permits one to factor characteristic polynomials using techniques for chromatic polynomials of signed graphs. Saito~\cite{s:tldflvf} and Terao~\cite{t:gefahstbf} studied a module of derivations associated with a hyperplane arrangement. When this module is free, the characteristic polynomial has roots which are the degrees of its basis elements.

Our method for factoring the characteristic polynomial is based on two simple results given in the next well-known lemma.

\begin{lem} \label{chiLem}
Let $P$ and $Q$ be posets. Then we have the following.
\begin{enumerate}
\item $\chi(P\times Q,t) = \chi(P,t)\chi(Q,t)$.
\item If $P\cong Q$, then $\chi(P,t) = \chi(Q,t)$. 
\end{enumerate}
\end{lem}

Now let us investigate a family of lattices whose characteristic polynomials have only nonnegative integer roots. We will often refer back to this example in the sequel. The \emph{partition lattice}, $\Pi_n$, is the lattice whose elements are the set partitions $\pi=B_1/\dots/B_k$ of $\{1,2,\dots, n\}$ under the refinement ordering. The subsets $B_i$ in a partition are called {\em blocks}. It is well-known that in this case the characteristic polynomial is given by 
\begin{displaymath}
\chi(\Pi_n, t) = (t-1)(t-2)\cdots (t-n+1).
\end{displaymath}

Note that the characteristic polynomial of the partition lattice can be written as the product of linear factors whose roots are in $\mathbb{Z}_{\ge0}$. Motivated by this fact, we consider a family of posets each having a single linear factor as its characteristic polynomial.

\begin{defi}
The \emph{claw} with $n$ atoms is the poset with a $\hat{0}$, $n$ atoms and no other elements. It will be denoted $CL_n$ and is the poset which has Hasse diagram depicted in Figure~\ref{clawFig}. Clearly,
\begin{displaymath}
\chi(CL_n,t) = t-n.
\end{displaymath}

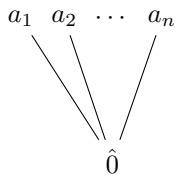
\begin{figure}
\begin{center}
\begin{tikzpicture}
\tikzstyle{elt}=[rectangle]
\matrix{
\node(1)[elt]{$a_1$};& \node(2)[elt]{$a_2$};&\node(i)[elt]{$\cdots$};&;&\node(n)[elt]{$a_n$};&\
\\[40pt]
&&;\node(0)[elt]{$\hat{0}$};&;\\
};
\draw(0)--(1);
\draw(0)--(2);
\draw(0)--(n);
\end{tikzpicture}
\end{center}
\caption{Claw with $n$ atoms}\label{clawFig}
\end{figure}
\end{defi}

Now let us look at the special case of $\Pi_3$. We wish to show that
\begin{displaymath}
\chi(\Pi_3, t) = (t-1)(t-2).
\end{displaymath}
Since the roots of $\chi(\Pi_3,t)$ are $1$ and $2$, we consider $CL_1\times CL_2$ which, by the first part of Lemma~\ref{chiLem}, has the same characteristic polynomial. Unfortunately, 
these two posets are not isomorphic since one contains a maximum element and the other does not. We now wish to modify $CL_1\times CL_2$ without changing its characteristic polynomial and in such a way that the resulting poset will be isomorphic to $\Pi_3$. It will then follow from the second part of Lemma~\ref{chiLem} that 
\begin{displaymath}
\chi(\Pi_3,t)=\chi(CL_1\times CL_2) = (t-1)(t-2).
\end{displaymath}
Let $CL_1$ have its atom labeled by $a$ and let $CL_2$ have its two atoms labeled by $b$ and $c$. Now suppose that we identify $(a,b)$ and $(a,c)$ in $CL_1\times CL_2$ and call this new element $d$. After this collapse, we get a poset isomorphic to $\Pi_3$ as can be seen in Figure~\ref{pi3Fig}. Note that performing this collapse did not change the characteristic polynomial since $\mu(d) = \mu((a,b)) +\mu((a,c))$ and $\rho(d) = \rho((a,b))=\rho((a,c))$. Thus we have fulfilled our goal.

\begin{figure}[ht]
\begin{center}
\begin{tikzpicture}
\tikzstyle{elt}=[rectangle]
\matrix{
&\node(4)[elt]{$123$};&\\\\[20pt] 
\node(1)[elt]{$12/3$};& \node(2)[elt]{$13/2$};&\node(3)[elt]{$1/23$};&\\\\[20pt]
&;\node(0)[elt]{$1/2/3$};&;\\
&;\node(space){\vspace{1in}};&;\\
&;\node(text1){$\Pi_3$};&;\\
};
\draw(0)--(1);
\draw(0)--(2);
\draw(0)--(3);
\draw(1)--(4);
\draw(2)--(4);
\draw(3)--(4);
\end{tikzpicture}
\hspace{.5 in}
\begin{tikzpicture}
\tikzstyle{elt}=[rectangle]
\matrix{
\node(ab){$(a,b)$};&&\node(ac){$(a,c)$};\\\\[20pt] 
\node(b0){$(\hat{0},b)$}; &\node(a0){$(a,\hat{0})$}; &;\node(c0){$(\hat{0},c)$};\\\\[20pt]
&;\node(0)[elt]{$(\hat{0},\hat{0})$};&;\\
&;\node(space){\vspace{1in}};&;\\
&;\node(text1){$CL_1\times CL_2$};&;\\
};
\draw(0)--(a0);
\draw(0)--(b0);
\draw(0)--(c0);
\draw(a0)--(ab);
\draw(a0)--(ac);
\draw(b0)--(ab);
\draw(c0)--(ac);
\end{tikzpicture}
\hspace{.5 in}
\begin{tikzpicture}
\tikzstyle{elt}=[rectangle]
\matrix{
&\node(ab){$(a,b)\sim (a,c)=d$};&\\\\[20pt] 
\node(b0){$(\hat{0},b)$}; &\node(a0){$(a,\hat{0})$}; &;\node(c0){$(\hat{0},c)$};\\\\[20pt]
&;\node(0)[elt]{$(\hat{0},\hat{0})$};&;\\
&;\node(space){\vspace{1in}};&;\\
&;\node(text1){$CL_1\times CL_2$ after identifying $(a,b)$ and $(a,c)$ };&;\\
};
\draw(0)--(a0);
\draw(0)--(b0);
\draw(0)--(c0);
\draw(a0)--(ab);
\draw(a0)--(ab);
\draw(b0)--(ab);
\draw(c0)--(ab);
\end{tikzpicture}

\caption{Hasse diagrams for the partition lattice example}\label{pi3Fig}
\end{center}
\end{figure}
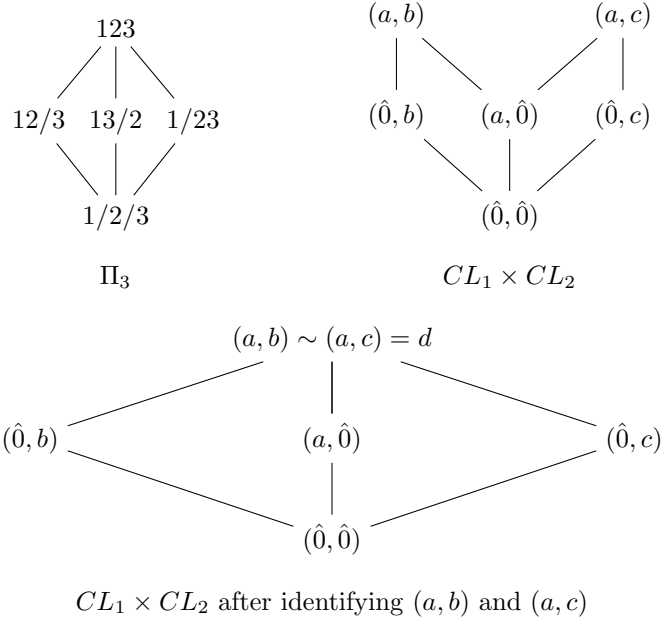

It turns out that we can use this technique of collapsing elements to find the roots of a characteristic polynomial in a wide array of posets, $P$. The basic idea is that it is trivial to calculate the characteristic polynomial of a product of claws. Moreover, under certain conditions which we will see later, we are able to identify elements of the product and form a new poset without changing the characteristic polynomial. If we can show the product with identifications made is isomorphic to $P$, then we will have succeeded in showing that $\chi(P,t)$ has only nonnegative integer roots.

In the next section, we formally define what it means to identify elements of a poset $P$ as well as give conditions under which making these identifications will not change the characteristic polynomial. In Section~\ref{ser}, we discuss a canonical way to put an equivalence relation on $P$ when it is a lattice and give three simple conditions which together imply that $\chi(P,t)$ has nonnegative integral roots. Section~\ref{rt} contains a generalization of the notion of a claw. This enables us to remove one of the conditions needed to prove factorization and we obtain our main result, Theorem~\ref{bigThm}. Section~\ref{pic} is concerned with partitions of the atom set of $P$ induced by a multichain. With one extra assumption, this permits us to give three conditions which are equivalent to $\chi(P,t)$ having the sizes of the blocks of the partition as roots; see Theorem~\ref{equivPropFactorSemimod}. This result will imply Stanley's Supersolvability Theorem~\cite{s:sl}. In section~\ref{agt} we will use Theorem~\ref{equivPropFactorSemimod} to prove a new theorem about the generating function for increasing spanning forests of a graph. We end with a section about open questions and future work.

\section{Quotients of Posets}
We begin this section by defining, in a rigorous way, what we mean by collapsing elements in a Hasse diagram of a poset. We do so by putting an equivalence relation on the poset and then ordering the equivalence classes. 

\begin{defi}\label{quoPosetDef}
Let $P$ be a poset and let $\sim$ be an equivalence relation on $P$. We define the {\em quotient} $P/\sim$ to be the set of equivalence classes with the binary relation $\leq$ defined by $X\leq Y$ in $P/\sim$ if and only if $x\leq y$ in $P$ for some $x\in X$ and some $y\in Y$. 
\end{defi}
 
 Note that this binary relation on $P/\sim$ is reflexive, although the antisymmetry and transitivity laws need not hold. For example let $P$ be the poset with chains $\hat{0}<x<y$ and $\hat{0}<w<z$ and no other relations. First, suppose that $A= \{w, x\}$ and $B=\{\hat{0},y, z\}$.  Then $A\leq B$ since $w<z$ and $B\leq A$ since $\hat{0}<w$.  However, $A\neq B$ and so the relation is not antisymmetric.  To see why it is not always transitive,  let $A=\{x\}$, $B=\{w,y\}$ and $C=\{z\}$.  Then $A\leq B$ since $x\leq y$ and $B\leq C$ since $w\leq z$, but $A\not\leq C$ since $x\not\leq z$.   Since we want the quotient to be a poset, it is necessary to require two more properties of our equivalence relation.

\begin{defi}\label{homogenQuoDefi}
Let $P$ be a poset and let $\sim$ be an equivalence relation on $P$. Order the equivalence classes as in the previous definition. We say the poset $P/\sim$ is a \emph{homogeneous quotient} if
\begin{itemize}
\item[(1)] $\hat{0}$ is in an equivalence class by itself, and 
\item[(2)] if $X\leq Y$ in $P/\sim$, then for all $x\in X$ there is a $y\in Y$ such that $x\leq y$.
\end{itemize}
\end{defi}

\begin{lem}
If $P$ is a poset and $P/\sim$ is a homogeneous quotient, then $P/\sim$ is a poset.
\end{lem}
\begin{proof}
As previously mentioned, the fact that $\leq$ in $P/\sim$ is reflexive is clear. To see why it is antisymmetric, suppose that $X\leq Y$ and $Y\leq X$. By definition, there is an $x\in X$ and $y\in Y$ with $x\leq y$. Since $Y\leq X$ there is an $x'\in X$ with $x\leq y\leq x'$. Since $X\leq Y$ there is a $y' \in Y$ with $x\leq y\leq x'\leq y'$. Continuing, we get a chain 
$$
x\leq y\leq x'\leq y'\leq \dots
$$
If any of the inequalities are equalities then we are done since the equivalence classes partition $P$. If all are strict, then we would have an infinite chain in $P$, but this contradicts the fact that $P$ is finite. Therefore it must be that $X=Y.$

For transitivity, suppose that $X\leq Y$ and $Y\leq Z$.  Since $X\leq Y$, there is some $x\in X$ and $y\in Y$ with $x\leq y$.  Moreover, since $Y\leq Z$ and our quotient is homogeneous, there is some $z\in Z$ with $y\leq z$.  It follows that $x\leq y\leq z$ and so $X\leq Z$.
\end{proof}

Since we would like to use quotient posets to find characteristic polynomials, it would be quite helpful if the M\"obius value of an equivalence class was the sum of the M\"obius values of the elements of the equivalence class. This is not always the case when using homogeneous quotients, however we only need one simple requirement on the equivalence classes so that this does occur. Note the similarity of the hypothesis in the next result to the definition of the M\"obius function. In what follows we will use $\mu(x)$ to denote the M\"obius value of $x\in P$ and $\mu(X)$ to denote the M\"obius value of the equivalence class $X\in P/\sim$.

\begin{lem}\label{sumLem}
Let $P/\sim$ be a homogeneous quotient poset. Suppose that for all nonzero $X\in P/\sim$,
\begin{equation}\label{sumCondEq}
\sum_{y \in L(X)} \mu(y) = 0
\end{equation}
where $L(X)$ is the lower order ideal generated by $X$ in $P$. 
Then, for all equivalence classes $X$
\begin{displaymath}
\mu(X) = \sum_{x\in X} \mu(x).
\end{displaymath}
\end{lem}

\begin{proof}
We induct on the length of the longest $\hat{0}$--$X$ chain to prove the result. If the length is zero, then $X =\hat{0}$. Since $P/\sim$ is a homogeneous quotient, there is only one element in $X$ and it is $\hat{0}$. The M\"obius value of the minimum of any poset is 1 and so the base case holds.

Now suppose that the length is positive. Then $X \neq \hat{0}$ and so by assumption,
$$
\sum_{y \in L(X)} \mu(y) =0.
$$
Breaking this sum into two parts and moving one to the other side of the equation gives 
\begin{equation}\label{sumOfMu}
\sum_{x\in X}\mu(x)=-\sum_{y\in L(X) \setminus X} \mu(y).
\end{equation}

Using the definition of $\mu$ and the induction hypothesis, we have that
$$
\mu(X) = -\sum_{Y<X}\mu(Y) = -\sum_{Y<X} \left(\sum_{y\in Y }\mu(y)\right).
$$
Since $P/\sim$ is a homogeneous quotient poset, we have that if $Y<X$ then for every $y\in Y$ there is an $x \in X$ with $y<x$. Therefore the previous sum ranges over all $y$ such that there is an $x\in X$ with $y<x$. Thus $y\in L(X)\setminus X$. 
Conversely, for each $y \in L(X)\setminus X$ there is an $x\in X$ with $y<x$. By the definition of $\leq$ in $P/\sim$, we have that this implies $Y <X$ where $Y$ is the equivalence class of $y$. It follows that
\begin{equation}\label{sumOfMu2}
\mu(X) = - \sum_{y\in L(X)\setminus X}\mu(y).
\end{equation} 
Combining this equation with~\ree{sumOfMu} completes the proof.
\end{proof}

For the remainder of the paper, we shall refer to the condition given by equation~\ree{sumCondEq} as the \emph{summation condition}. From the previous lemma, we know how the M\"obius values behave when taking quotients under certain circumstances. We also need to know how the rank behaves under quotients. As we did earlier for the M\"obius function, we will use $\rho(x)$ for the rank of $x\in P$ and $\rho(X)$ for the rank of the equivalence class $X\in P/\sim$.

\begin{lem}\label{rankedQuotientLem}
Let $P/\sim$ be a homogeneous quotient poset. Suppose that for all $x,y\in P$, $x\sim y$ implies $\rho(x)=\rho(y)$. Then $P/\sim$ is ranked and $\rho(X) = \rho(x)$ for all $x\in X$.
\end{lem}
\begin{proof}
We actually prove a stronger result. We claim that $X\cover Y$ (where $\cover$ denotes a covering relation) implies there is an $x\in X$ and a $y\in Y$ such that $x\cover y$. To see why this implies the lemma, suppose that there were two chains $\hat{0}=X_1\cover X_2 \cover \cdots \cover X_n$ and $\hat{0}=Y_1\cover Y_2 \cover \cdots \cover Y_m$ with $X_n=Y_m$. Then for the corresponding chains $\hat{0}=x_1\cover x_2 \cover \cdots \cover x_n$ and $\hat{0}=y_1\cover y_2 \cover \cdots \cover y_m$ we have that $\rho(x_n)=\rho(y_m)$ since elements in the same equivalence class have the same rank. This forces $n=m$ and so $P/\sim$ must be ranked. Additionally, it is easy to see that this implies that $\rho(X)=\rho(x)$ for all $x\in X$.

To prove our claim, note that by the definition of a homogeneous quotient, if $X\cover Y$ then there is an $x\in X$ and $y\in Y$ with $x<y$. Suppose that there was some $z\in P$ with $x<z<y$. Then $\rho(x)<\rho(z)<\rho(y)$ and $X\leq Z\leq Y$ where $Z$ is the equivalence class of $z$. Since all elements in an equivalence class have the same rank this implies that $X<Z<Y$ in $P/\sim$, which contradicts the fact that $Y$ covered $X$.
\end{proof}

Applying Lemma~\ref{sumLem}, Lemma~\ref{rankedQuotientLem} and the definition of the characteristic polynomial we immediately get the following corollary.

\begin{cor}\label{chiQuoCor}
Let $P/\sim$ be a homogeneous quotient. If the summation condition~\ree{sumCondEq} holds for all nonzero $X\in P/\sim$, and $x\sim y$ implies $\rho(x)=\rho(y)$, then
\begin{displaymath}
\chi(P/\sim,t)=\chi(P,t).
\end{displaymath}
\end{cor}

We now have conditions under which the characteristic polynomial does not change when taking a quotient. 
However, the previous results do not tell us how to choose an appropriate equivalence relation for a given poset. It turns out that when the poset is a lattice, there is a canonical choice for $\sim$, as we will see in the next section.

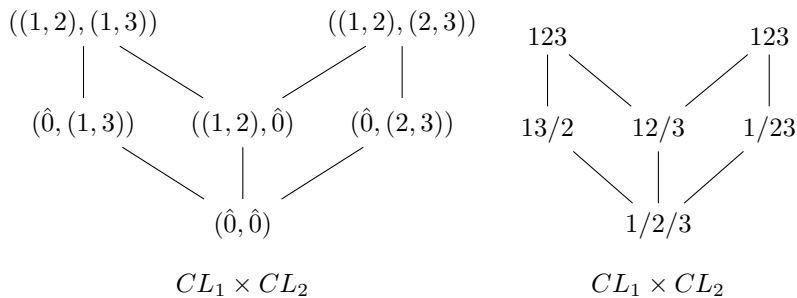
\begin{figure}
\begin{center}
\begin{tikzpicture}
\tikzstyle{elt}=[rectangle]
\matrix{
\node(ab){$((1,2),(1,3))$};&&\node(ac){$((1,2),(2,3))$};\\\\[20pt] 
\node(b0){$(\hat{0},(1,3))$}; &\node(a0){$((1,2),\hat{0})$}; &;\node(c0){$(\hat{0},(2,3))$};\\\\[20pt]
&;\node(0)[elt]{$(\hat{0},\hat{0})$};&;\\
&;\node(space){\vspace{1in}};&;\\
&;\node(text1){$CL_1\times CL_2$};&;\\
};
\draw(0)--(a0);
\draw(0)--(b0);
\draw(0)--(c0);
\draw(a0)--(ab);
\draw(a0)--(ac);
\draw(b0)--(ab);
\draw(c0)--(ac);
\end{tikzpicture}
\begin{tikzpicture}
\tikzstyle{elt}=[rectangle]
\matrix{
\node(ab){$123$};&&\node(ac){$123$};\\\\[20pt] 
\node(b0){$13/2$}; &\node(a0){$12/3$}; &;\node(c0){$1/23$};\\\\[20pt]
&;\node(0)[elt]{$1/2/3$};&;\\
&;\node(space){\vspace{1in}};&;\\
&;\node(text1){$CL_1\times CL_2$};&;\\
};
\draw(0)--(a0);
\draw(0)--(b0);
\draw(0)--(c0);
\draw(a0)--(ab);
\draw(a0)--(ac);
\draw(b0)--(ab);
\draw(c0)--(ac);
\end{tikzpicture}
\end{center}

\caption{Hasse diagrams for partition lattice example with new labelings}\label{pi3FigNewLabel}

\end{figure}

\section{The Standard Equivalence Relation}
\label{ser}

Let us look at the partition lattice example again and give new labelings to $CL_1\times CL_2$ which will be helpful in determining an equivalence relation. First, we set up some notation for the atoms of the partition lattice. For $i<j$, let $(i ,j)$ \ denote the atom which has $i$ and $j$ in one block and all other elements in singleton blocks. Let $CL_1$ have its atom labeled by $(1,2)$ and $CL_2$ have its atoms labeled by $(1,3)$ and $(2,3)$. In both of the claws, label the minimum element by $\hat{0}$. The poset on the left in Figure~\ref{pi3FigNewLabel} shows the induced labeling on $CL_1\times CL_2$.

Now relabel $CL_1\times CL_2$ by taking the join in $\Pi_3$ of the two elements in each pair. The poset on the right in Figure~\ref{pi3FigNewLabel} shows this step. Finally, identify elements which have the same label. In this case, this means identifying the top two elements as we did before. Upon doing this, we get a poset which is isomorphic to $\Pi_3$ and has the same labeling as $\Pi_3$. 

In order to generalize the previous example, we will be putting an equivalence relation on the product of claws whose atom sets come from partitioning the atoms of the original lattice. We need some terminology before we can define our equivalence relation. 

Suppose that $L$ is a lattice and $(A_1, A_2, \dots, A_n)$ is an ordered partition of the atoms of $L$. We will use $CL_{A_i}$ to denote the claw whose atom set is $A_i$ and whose minimum element is labeled by $\hat{0}_L$ (or just $\hat{0}$ if $L$ is clear from context). The elements of $\prod_{i=1}^n CL_{A_i}$ will be called \emph{atomic transversals} and written in boldface. (The reason for the adjective ``atomic" is because we will be considering more general transversals in Section~\ref{rt}.) Since the rank of an element in the product of claws is just the number of nonzero elements in the tuple, it will be useful to have a name for this number. 
For $\bfm{t} \in~\prod_{i=1}^n CL_{A_i}$ define the \emph{support} of $\bfm{t}$ as the number of nonzero elements in the tuple $\bfm{t}$. We will denote it by $\supp \bfm{t}$.

We will use the notation $\bfm{t}(e^{i})$ to denote the ordered tuple obtained by replacing the $i^{th}$ coordinate of $\bfm{t} =(t_1,t_2,\dots,t_n)$ with an element $e$. That is, 
$$
\bfm{t}(e^{i}) =(t_1, t_2, \dots, t_{i-1}, e, t_{i+1}, \dots, t_n).
$$
We will also need a notation for the join of the elements of $\bfm{t}$ which will be
\begin{displaymath}
\bigvee \bfm{t} = t_1\vee t_2 \vee \dots \vee t_n.
\end{displaymath}

With this new terminology we are now in a position to define a natural equivalence relation on the product of the claws. Since we are trying to show that the characteristic polynomial of a lattice has certain roots, we will need to show that the quotient of the product of claws is isomorphic to the lattice. Therefore it is reasonable to define the equivalence relation by identifying two elements of the product of claws if their joins are the same in $L$. 

\begin{defi}
Let $L$ be a lattice and let $(A_1, A_2, \dots, A_n)$ be an ordered partition of the atoms of $L$. The \emph{standard equivalence relation} on $\prod_{i=1}^n CL_{A_i}$ is defined by
\begin{displaymath}
\bfm{s} \sim \bfm{t} \mbox{ in }\prod_{i=1}^n CL_{A_i} \iff \bigvee \bfm{s} =\bigvee \bfm{t} \mbox{ in } L.
\end{displaymath}
\end{defi}

We will use the notation 
\begin{displaymath}
\mathcal{T}_x^A= \left\{\bfm{t}\in\prod_{i=1}^n CL_{A_i} : \bigvee \bfm{t}=x\right \}
\end{displaymath}
and call the elements of this set \emph{atomic transversals of $x$}. Therefore, the equivalence classes of the quotient $\left(\prod_{i=1}^n CL_{A_i}\right)/\sim$ are of the form $\mathcal{T}_x^A$ for some $x\in L$. It is obvious that the standard equivalence relation is an equivalence relation. To be able to use any of the theorems from the previous section, we need to make sure that taking the quotient with respect to the standard equivalence relation gives us a homogeneous quotient. Moreover, we will need a way to determine if the summation condition~\ree{sumCondEq} holds for all nonzero elements of the quotient. We do this in the next lemma. For the rest of the paper we will use the notation $A_x$ for the set of atoms below $x$.

\begin{lem}\label{uniformRankImpliesHomogenQuotLem}
Let $L$ be a lattice and let $(A_1, A_2, \dots, A_n)$ be an ordered partition of the atoms of $L$. Let $\sim$ be the standard equivalence relation on $\prod_{i=1}^n CL_{A_i}$. 
Suppose that the following hold. 
\begin{enumerate}
\item[(1)] For all $x\in L$, $\mathcal{T}_x^A \neq \emptyset$. 
\item[(2)] If $\bfm{t}\in \mathcal{T}_x^A$, then $|\supp \bfm{t}|=\rho(x)$.
\end{enumerate}
Under these conditions, 
\begin{itemize}
\item[(a)] The lower order ideal generated by the set $\mathcal{T}_x^A$ in $\prod_{i=1}^n CL_{A_i}$ is given by 
$$
L(\mathcal{T}_x^A) =\{\bfm{t} : t_i\leq x \mbox{ for all } i\}.
$$
\item[(b)] The quotient $\left(\prod_{i=1}^n CL_{A_i}\right)/\sim$ is homogeneous.
\item[(c)] For all nonzero $\mathcal{T}_x^A \in \left(\prod_{i=1}^n CL_{A_i}\right)/\sim$ the summation condition~\ree{sumCondEq} holds if and only if for each  nonzero  $x\in L$ there is an $i$ such that $|A_i \cap A_x|=1$.
\end{itemize}
\end{lem}

\begin{proof} First, we show (a).
We claim that assumptions (1) and (2) imply that if $a\in A_x$ then there is an atomic transversal for $x$ which contains $a$. To verify the claim, use assumption (1) to pick $\bfm{t} \in \mathcal{T}_x^A$ and let $\bfm{r} = \bfm{t}(a^i)$. By construction and assumption (2), $\rho(\bigvee \bfm{r}) = |\supp \bfm{r}| \geq |\supp \bfm{t}| =\rho(x)$. But also $\bigvee \bfm{r} \leq x$ which forces $\bigvee \bfm{r} =x$. Thus $a$ is in the atomic transversal $\bfm{r}$ for $x$.

The definition of $\mathcal{T}_x^A$ gives us the inclusion $ L(\mathcal{T}_x^A) \subseteq\{\bfm{t} : t_i\leq x \mbox{ for all } i\}$. The reverse inclusion holds by the previous claim.

Next, we verify (b). It is clear that $\bfm{t}\in \mathcal{T}_{\hat{0}}^A$ if and only if $\bfm{t}= (\hat{0},\hat{0},\dots, \hat{0})$ and so part (1) of Definition~\ref{homogenQuoDefi} is satisfied. To show part (2), suppose that $\mathcal{T}_x^A\leq \mathcal{T}_y^A$ as in Definition~\ref{quoPosetDef}. Then there is some $\bfm{t}\in\mathcal{T}^A_x$ and $\bfm{s}\in\mathcal{T}_y^A$ with $\bfm{t}\leq \bfm{s}$. It follows that $\bigvee \bfm{t}\leq \bigvee\bfm{s}$ and so $x\leq y$.  Let $\bfm{t}\in\mathcal{T}_x^A$. Using the fact that $t_i\leq x\leq y$ and part (a), we have that $\bfm{t}\in L(\mathcal{T}_y^A)$. It follows that there is some $\bfm{s}\in \mathcal{T}_y^A$ with $\bfm{t}\leq \bfm{s}$ and so part (2) of Definition~\ref{homogenQuoDefi} holds.

Finally, we demonstrate (c). Fix $x\in L$ and let $N_i$ be the number of atoms below $x$ in $A_i$. Let $I$ be the set of indices $i$ such that $N_i>0$. By relabeling, if necessary, we may assume that $I=\{1,2, \dots, k\}$. It follows from part (a) that the number of atomic transversals in $ L(\mathcal{T}_x^A)$ with support size $i$ is $e_{i}(N_1, N_2, \dots, N_k)$ where $e_i$ is the $i^{th}$ elementary symmetric function. For each atomic transversal $\bfm{t} \in L(\mathcal{T}_x^A)$ we have that $\mu(\bfm{t}) = (-1)^{|\supp \bfm{t}|}$. Therefore, 
$$
\sum_{\bfm{t}\in L(\mathcal{T}_x^A)} \mu(\bfm{t}) = \sum_{i=0}^k (-1)^i e_i(N_1, N_2, \dots, N_k)= \prod_{i=1}^k(1-N_i).
$$
Therefore the summation condition~\ree{sumCondEq} holds for each nonzero element in the quotient if and only if for each nonzero $x\in L$ there is an index $i$ such that $|A_i\cap A_x|=1$. 
\end{proof}

Combining the previous result with Corollary~\ref{chiQuoCor} gives conditions under which the product of claws and its quotient have the same characteristic polynomial. We also need to show that there is an isomorphism between $L$ and this quotient. This will give us the desired factorization.

\begin{thm}\label{bigThmAtomicVer}
Let $L$ be a lattice and let $(A_1, A_2, \dots, A_n)$ be an ordered partition of the atoms of $L$. Let $\sim$ be the standard equivalence relation on $\prod_{i=1}^n CL_{A_i}$. Suppose the following hold.
\begin{enumerate}
\item[(1)] For all $x\in L$, $\mathcal{T}_x^A \neq \emptyset$.
\item[(2)] If $\bfm{t}\in \mathcal{T}_x^A$, then $|\supp \bfm{t}|=\rho(x)$.
\item[(3)] For each nonzero $x\in L$ there is some $i$ with $|A_i \cap A_x |=1$. 
\end{enumerate}
Then we can conclude the following. 
\begin{itemize}
\item[(a)] For all $x\in L$, $\mu(x) = (-1)^{\rho(x)}|\mathcal{T}_x^A|$.
\item[(b)] $\chi(L,t) = \displaystyle\prod_{i=1}^n (t-|A_i|)$.
\end{itemize}
\end{thm}

\begin{proof}
Let $P= \prod_{i=1}^n CL_{A_i}$. First, we show that $L\cong P/\sim$. Define a map $\varphi: (P/\sim) \hspace{5 pt} \rightarrow L$  by $\varphi(\mathcal{T}_x^A) =x$. It is easy to see that $\varphi$ is well-defined. Define 
$\psi: L \rightarrow ( P/\sim)$ by $\psi(x) = \mathcal{T}_x^A$. By assumption $\mathcal{T}_x^A\neq \emptyset$ and so $\psi$ is well-defined. Moreover, it is clear that $\varphi$ and $\psi$ are inverses of each other.

To show that $\varphi$ is order preserving, suppose that $\mathcal{T}_x^A \leq \mathcal{T}_y^A$. Then just as in the proof of Lemma~\ref{uniformRankImpliesHomogenQuotLem} part (b), we have that $x\leq y$ and so $\varphi$ is order preserving.

To show that $\psi$ is order preserving, suppose that $x\leq y$. Then applying the same technique as in the proof of Lemma~\ref{uniformRankImpliesHomogenQuotLem} part (b) we get that there is a $\bfm{t}\in\mathcal{T}_x^A$ and $\bfm{s}\in\mathcal{T}_y^A$ with $\bfm{t}\leq \bfm{s}$. By the definition of $\leq$ in $P/\sim$ we get that $\mathcal{T}_x^A\leq \mathcal{T}_y^A$ and so $\psi$ is order preserving.

To obtain (a), note that the M\"obius value of an element in the product of claws is $\mu(\bfm{t})=(-1)^{|\supp \bfm{t}|}$. Therefore, using Lemma~\ref{sumLem}, we get
$$
\mu(\mathcal{T}_x^A) = \sum_{\bfm{t}\in \mathcal{T}_x^A} \mu(\bfm{t})= \sum_{\bfm{t}\in \mathcal{T}_x^A} (-1)^{|\supp \bfm{t}|}.
$$
Using the isomorphism between $L$ and the quotient as well as the fact that, by assumption (2), all the atomic transversals for $x$ have size $\rho(x)$, we have 
$$
\mu(x) = \mu(\mathcal{T}_x^A)= (-1)^{\rho(x)}|\mathcal{T}_x^A|
$$
as desired.

Finally, to verify (b) apply Corollary~\ref{chiQuoCor} and Lemma~\ref{uniformRankImpliesHomogenQuotLem} to get that
$$
\prod_{i=1}^n (t-|A_i|)=\chi(P,t) = \chi(P/\sim,t).
$$
Now part (b) follows immediately since $L\cong P/\sim$.
\end{proof}

\begin{example}\label{piExample}
Let us return to the partition lattice $\Pi_n$ and see how we can apply Theorem~\ref{bigThmAtomicVer}. Label the atoms $(i,j)$ as in the beginning of section~\ref{ser}. Partition the atoms as $(A_1,A_2, \dots, A_{n-1})$ where 
$$
A_j=\{(i, j+1)\mid i<j+1\}.
$$ 
With each atomic transversal $\bfm{t}$ we will associate a graph, $G_{\bfm{t}}$ on $n$ vertices such that there is an edge between vertex $i$ and vertex $j$ if and only if $(i,j)$ is in $\bfm{t}$. We will use the graph to verify the assumptions of Theorem~\ref{bigThmAtomicVer} are satisfied for $\pi\in \Pi_n$.

First, let us show assumption (1) of Theorem~\ref{bigThmAtomicVer}  holds. In the case where $\pi \in \Pi_n$ consists of a single block $B=\{b_1<b_2<\dots<b_m\}$, the elements $(b_1,b_2), (b_2,b_3),\dots, (b_{m-1},b_m)$ form the non-trivial elements of an atomic transversal whose join is $B$. Now to get the elements which have more than one nontrivial block, follow the same procedure for each block and use the transversal which corresponds to the union of the transversals of the blocks considered as sets.  It follows every element has an atomic transversal.

Next, we prove that assumption (2) holds. We claim that if $\bfm{t} \in \mathcal{T}_\pi^A$ then $G_{\bfm{t}}$ is a forest. To see why, suppose that there was a cycle and let $c$ be the largest vertex in the cycle. Then $c$ must be adjacent to two smaller vertices $a$ and $b$ which implies that both $(a,c)$ and $(b,c)$ must be in $\bfm{t}$. This is impossible since both come from $A_{c-1}$. 

Since $G_{\bfm{t}}$ is a forest, if $G_{\bfm{t}}$ has $k$ components then the number of edges in $G_{\bfm{t}}$ is $n-k$. It is not hard to see that $i$ and $j$ are in the same block in $\bigvee \bfm{t}$ if and only if $i$ and $j$ are in the same component of $G_{\bfm{t}}$. Moreover, it is well known that if $\pi\in \Pi_n$ and $\pi$ has $k$ blocks then $\rho(\pi)=n-k$. It follows that if $\bfm{t}\in \mathcal{T}_\pi^A$ and $\pi$ has $k$ blocks then $ |\supp \bfm{t}| =|E(G_{\bfm{t}})|= n-k=\rho(\pi)$. We conclude that assumption (2) holds.

Finally, to verify assumption (3), let $\pi\in \Pi_n$ with $\pi\neq \hat{0}$. Then $\pi$ contains a nontrivial block. Let $i$ be the second smallest number in this block. We claim that there is only one atom in $A_{i-1}$ below $\pi$. First note that there is some atom below $\pi$ in $A_{i-1}$ namely $(a,i)$ where $a$ is the smallest element of the block. Suppose there was more than one atom below $\pi$ in $A_{i-1}$ and let $(a,i), (b,i)\in A_{i-1}$ with $(a,i), (b,i)\leq\pi$. Then $(a,i)\vee (b,i)\leq \pi$ and so $a$, $b$ and $i$ are all in the same block in $\pi$ which is impossible since $a,b<i$ but $i$ was chosen to be the second smallest in its block.

Now applying the theorem we get that 
\begin{displaymath}
\chi(\Pi_n,t) = (t-1)(t-2)\cdots(t-n+1)
\end{displaymath}
since $|A_i|=i$ for $1\le i\le n-1$.

We should note that it is not trivial to find a partition of the atom set which satisfies the conditions of Theorem~\ref{bigThmAtomicVer}.  However, for certain lattices (including $\Pi_n$) there is a canonical choice for the partition.  This is described in more detail in section~\ref{pic}.
\end{example}

  Theorem~\ref{bigThmAtomicVer} can be used to prove  Terao's result~\cite{t:fosa} about the characteristic polynomial of a hyperplane arrangement with a nice partition.  In fact the notion of a nice partition is the combination of assumptions (2) and (3) of Theorem~\ref{bigThmAtomicVer} in the special case of a central hyperplane arrangement.

\section{Rooted Trees}
\label{rt}

One of the drawbacks of Theorem~\ref{bigThmAtomicVer} is that assumption (1) requires that every element of the lattice is \emph{atomic} meaning that is a join of atoms.   In this case $L$ is said to be \emph{atomic}. However, by generalizing the notion of a claw to that of a rooted tree, we will be able to remove this assumption and derive Theorem~\ref{bigThm} below which applies to a wider class of lattices. 
\begin{defi}
Let $L$ be a lattice and $S$ be a subset of $L$ containing $\hat{0}$. Let $\mathcal{C}$ be the collection of saturated chains of $L$ which start at $\hat{0}$ and use only elements of $S$. The \emph{rooted tree with respect to $S$} is the poset obtained by ordering $\mathcal{C}$ by inclusion and will be denoted by $RT_S$. 
\end{defi}

It is easy to see that given any subset $S$ of a lattice containing $\hat{0}$, the Hasse diagram of $RT_S$ always contains a $\hat{0}$ and has no cycles. This explains the choice of rooted tree for the name of the poset.

Strictly speaking the elements of $RT_S$ are chains of $L$. However, it will be useful to think of the elements of $RT_S$ as elements of $L$ where we associate a chain $C$ with its top element. One can still recover the full chain by considering the unique path from $\hat{0}$ to $C$ in $RT_S$.
Let us consider an example in $\Pi_3$. As before, partition the atom set as $A_1=\{12/3\}$ and $A_2=\{13/2,1/23\}$. Let $S_1, S_2$ be the upper order ideals generated by $A_1, A_2$, respectively, together with $\hat{0}$. Then we get $RT_{S_1}$ and $RT_{S_2}$ as in Figure~\ref{rootedTreePi3}. Note that we label the chains 
$\hat{0}< 12/3 < 123$, $\hat{0}< 13/2< 123$ and $\hat{0}<1/23 < 123$ in $S_1$ and $S_2$ all by $123$ in $RT_{S_1}$ and $RT_{S_2}$ since each of these chains terminates at $123$.

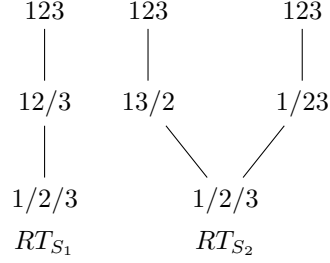
\begin{figure}
\begin{center}
\begin{tikzpicture}
\tikzstyle{elt}=[rectangle]
\matrix{\node(123){ $123$};\\\\[20pt] 
\node(12){$12/3$};\\[20pt] 
\node(zero){$1/2/3$};\\
\node(text){$RT_{S_1}$};\\
};
\draw(zero)--(12);
\draw(12)--(123);
\end{tikzpicture}
\begin{tikzpicture}
\tikzstyle{elt}=[rectangle]
\matrix{
\node(ab){$123$};&&\node(ac){$123$};\\\\[20pt] 
\node(b0){$13/2$}; & &;\node(c0){$1/23$};\\\\[20pt]
&;\node(0)[elt]{$1/2/3$};&;\\
&;\node(text1){$RT_{S_2}$};&;\\
};
\draw(0)--(b0);
\draw(0)--(c0);
\draw(b0)--(ab);
\draw(c0)--(ac);
\end{tikzpicture}
\end{center}
\caption{Hasse diagrams for rooted trees }\label{rootedTreePi3}
\end{figure}

In the previous sections, we used a partition of the atom set to form claws. In this section, we will use the partition of the atom set to form rooted trees. Given an ordered partition of the atoms of a lattice $(A_1,A_2, \dots,A_n)$, for each $i$ we form the rooted tree $RT_{\hat{U}(A_i)}$ where $\hat{U}(A_i)$ is the upper order ideal generated by $A_i$ together with $\hat{0}$. Note that since $(A_1,A_2,\dots,A_n)$ is a partition of the atoms, every element of the lattice appears in an $RT_{\hat{U}(A_i)}$ for some $i$. 

Given $(A_1,A_2,\dots,A_n)$, we call $\bfm{t}\in \prod_{i=1}^n RT_{\hat{U}(A_i)}$ a \emph{transversal}. We will use the notation,
$$
\mathcal{T}_x= \left\{\bfm{t}\in\prod_{i=1}^n RT_{\hat{U}(A_i)}: \bigvee \bfm{t}=x\right \}
$$
and call such elements \emph{transversals of $x$}. If $\bfm{t}$ consists of only atoms of $L$ or $\hat{0}$ then $\bfm{t}$ is called an \emph{atomic transversal}. This agrees with the terminology we used for claws. The set of atomic transversals for $x$ will be denoted $\mathcal{T}_x^A$ as before.

There is very little change in the approach  using rooted trees as opposed to claws. As before, given a partition $(A_1,A_2, \dots, A_n)$ of the atom set of $L$, we will put the standard equivalence relation on $ \prod_{i=1}^n RT_{\hat{U}(A_i)}$. Note that one can take the join using all the elements of a chain or just the top element as the results will be equal. Since we are using rooted trees, the natural map from $ \left(\prod_{i=1}^n RT_{\hat{U}(A_i)}\right)/\sim$ to $L$ is automatically surjective. In other words, we can remove the condition that every element of $L$ has an atomic transversal. Additionally, since the M\"obius function of a tree is zero everywhere except in $\hat{0}$ and its atoms, when we take the product of the trees, the M\"obius value of any transversal which is not atomic is zero and so does not affect $\chi$. Therefore, we get the following improvement on Theorem~\ref{bigThmAtomicVer}.

\begin{thm}\label{bigThm}
Let $L$ be a lattice and let $(A_1, A_2, \dots, A_n)$ be an ordered partition of the atoms of $L$. Let $\sim$ be the standard equivalence relation on $\prod_{i=1}^n RT_{\hat{U}(A_i)}$. Suppose the following hold: 
\begin{enumerate} 
\item[(1)] If $\bfm{t}\in \mathcal{T}_x^A$, then $|\supp \bfm{t}|=\rho(x)$.
\item[(2)] For each nonzero $x\in L$ there is some $i$ with $|A_i \cap A_x |=1.$
\end{enumerate}
Then we can conclude the following.
\begin{itemize}
\item[(a)] For all $x\in L$, $\mu(x) = (-1)^{\rho(x)}|\mathcal{T}_x^A|$.
\item[(b)] $\chi(L,t) = t^{\rho(L)-n}\displaystyle\prod_{i=1}^n (t-|A_i|)$.
\end{itemize}
\end{thm}

\begin{proof}
Let $P= \prod_{i=1}^n RT_{\hat{U}(A_i)}$. We need to show that $P/\sim$ is homogeneous. 
The first condition of the definition is obvious.  For the second, suppose that $\mathcal{T}_x\leq\mathcal{T}_y$ and  $\bfm{t}\in \mathcal{T}_x$.   It follows that $x\leq y$. We need to show that there exists some $\bfm{s}\in \mathcal{T}_y$ such that $\bfm{t}\leq \bfm{s}$.   Let $i$ be an index such that $A_i\cap A_y\neq \emptyset$ so that $y\in\hat{U}(A_i)$. If $\bfm{t}\in \mathcal{T}_x$, then $t_j\leq x\leq y$ for all $j$. Therefore, $\bfm{t}(y^i)\in \mathcal{T}_y$ and $\bfm{t}\leq \bfm{t}(y^i)$. It follows that $P/\sim$ is homogeneous.

In the proof of Theorem~\ref{bigThmAtomicVer}, we showed that the lattice and the quotient of the product of claws were isomorphic. The proof that $L$ and $P/\sim$ are isomorphic is essentially the same. If we define $\varphi$ and $\psi$ analogously, then the only difference is showing $\psi$ is order preserving in which case one can use the same ideas as in the previous paragraph to complete the demonstration.

Now we verify that the summation condition~\ree{sumCondEq} holds for all nonzero elements of $P/\sim$. We only need to modify the proof that we gave in Lemma~\ref{uniformRankImpliesHomogenQuotLem} part (c) slightly. Analogously to the proof of part (a) of that lemma, one sees that $L(\mathcal{T}_x) = \{\bfm{t} : t_i\leq x \mbox{ for all } i\}$.
Using this and the fact that only atomic transversals have nonzero M\"obius values, the proof of Lemma~\ref{uniformRankImpliesHomogenQuotLem} part (c) goes through as before with $\mathcal{T}_x^A$ replaced by $\mathcal{T}_x$.

Now applying Lemma~\ref{sumLem} and the fact that $\mu(\bfm{t})=0$ if $\bfm{t}$ is not atomic, we get
\begin{equation}\label{sumOfMuRootedTreesEq}
\mu(\mathcal{T}_x) = \sum_{\bfm{t} \in \mathcal{T}_x} \mu(\bfm{t})= \sum_{\bfm{t} \in \mathcal{T}_x^A} \mu(\bfm{t}).
\end{equation}
Then applying the same proof as in Theorem~\ref{bigThmAtomicVer} gives us (a).

To finish the proof we define a modification of the characteristic polynomial for any ranked poset $P$,
$$
\bar{\chi}(P,t) = \sum_{x\in P} \mu(x) t^{-\rho(x)}.
$$
We claim that $\bar{\chi}(P,t)=\bar{\chi}(P/\sim,t)$. Applying assumption (1) and the isomorphism $L\cong P/\sim$, we get that for every $\bfm{t}\in\mathcal{T}_x^A$ we have
$$
\rho(\bfm{t})=|\supp\bfm{t}|=\rho(x)=\rho(\mathcal{T}_x).
$$
This combined with equation~\ree{sumOfMuRootedTreesEq}, proves the claim.

Now if $RT$ is a rooted tree with $k$ atoms then $\bar{\chi}(RT,t)=t^{-1}(t-k)$. It follows that 
$$
\bar{\chi}(P,t)= t^{-n}\prod_{i=1}^n(t-|A_i|).
$$
Since $\bar{\chi}$ is preserved by isomorphism, 
$$
\bar{\chi}(L,t) =\bar{\chi}(P/\sim,t)=\bar{\chi}(P,t) = t^{-n}\prod_{i=1}^n(t-|A_i|).
$$
Multiplying by $t^{\rho(L)}$ gives us part (b).
\end{proof}

\section{Partitions Induced by a Multichain}
\label{pic}

It turns out that under certain circumstances we can show that assumption (2) of Theorem~\ref{bigThm} and factorization of the characteristic polynomial are equivalent. To be able to prove this equivalence, we will not be able to take an arbitrary partition of the atoms, but rather we will need the partition to be induced by a multichain in the lattice. 

If $L$ is a lattice and  $C: \hat{0}=x_0\leq x_1\leq\dots \leq x_n=\hat{1}$ is a $\hat{0}$--$\hat{1}$ multichain of $L$ we get an ordered partition $(A_1,A_2,\dots,A_n)$ of the atoms of $L$ by defining the set $A_i$ as
\begin{displaymath}
A_i = \{a \in A(L) \mid \hspace{ 2 pt} a\leq x_i \mbox{ and } a \nleq x_{i-1}\}.
\end{displaymath}
In this case we say $(A_1, A_2, \dots, A_n)$ is \emph{induced} by the multichain $C$. Note that we do not insist that our multichain  be a chain nor does it need to be saturated as is usually done in the literature. Partitions induced by multichains have several nice properties. The first property will apply to any lattice (Lemma~\ref{chainOneNi}), but for the second we will need the lattice to be semimodular (Lemma~\ref{chainIndSet}). Before we get to these properties, we need a modification of Lemma~\ref{sumLem}.

\begin{lem}\label{maxElemQuotient}
Suppose that $P/\sim$ is a homogeneous quotient and that for all non-maximal, nonzero $X \in P/\sim$ we have that 
$$
\sum_{y\in L(X)} \mu(y) =0
$$
Then for all $X \in P/\sim$ 
$$
\mu(X) =
\begin{cases}
\displaystyle\sum_{x\in X} \mu(x) & \text{if $X$ is not maximal,} \\\\ \displaystyle\sum_{x\in X} \mu(x) - \sum_{y\in L(X)}\mu(y) &\text{if $X$ is maximal. }
\end{cases}
$$
\end{lem}
\begin{proof}

If $X$ is not maximal, then the proof of Lemma~\ref{sumLem} goes through as before.

Now suppose that $X$ is maximal. If $X=\hat{0}$ then the result holds since $P/\sim$ is trivial.   So suppose $X\neq\hat{0}$. In the proof of Lemma~\ref{sumLem}, we derived equation~\ree{sumOfMu2} without using the summation condition~\ree{sumCondEq} and so it still holds. Moreover, it is easy to see that this equation is equivalent to the one for maximal $X$ in the statement of the current result.
\end{proof}

Given a lattice and a partition of the atoms, it will be useful to know when elements of a lattice do not satisfy condition (2) of Theorem~\ref{bigThm}. This is possible to do when the partition of the atoms is induced by a multichain.

\begin{lem}\label{chainOneNi}
Let $L$ be a lattice and let $(A_1, A_2, \dots, A_n)$ be induced by a multichain $C: \hat{0}=x_0\leq x_1\leq\dots \leq x_n=\hat{1}$.  Let $N_i$ be the number of atoms below an element $x\in L$ in $A_i$. If $N_i\neq 1$ for all $i$ and $x\neq \hat{0}$ is minimal with respect to this property, then for all but one $i$, $N_i=0$. 
\end{lem}
\begin{proof}
Suppose that $x$ is minimal, but that $N_i>1$ for at least two $i$. Let $k$ be the smallest index with $N_k\neq0$, and $B\subseteq A_k$ be the atoms below $x$ in $A_k$ so $|B|\ge2$. Let $y=\bigvee B$. So, by the choice of $B$, $y\le x_k$ which implies that the atoms below $y$ are in $A_i$ for $i\leq k$. So the choice of $A_k$ forces the set of atoms below $y$ to be $B$ which is a proper subset of the set of atoms below $x$, and thus $y<x$. Since $|B|\ge2$, this contradicts the choice of $x$.
\end{proof}

The next definition gives one of the conditions equivalent to factorization when the atom partition is induced by a multichain.

\begin{defi}
Let $L$ be a lattice and let $C: \hat{0}=x_0\leq x_1\leq\dots \leq x_n=\hat{1}$ be a $\hat{0}$--$\hat{1}$ multichain. For atomic $x\in L$, $x$ neither $\hat{0}$ nor an atom, let $i$ be the index such that $x\leq x_{i}$ but $x\not\leq x_{i-1}$. We say that $C$ satisfies the \emph{meet condition} if, for each such $x$, we have $x\wedge x_{i-1} \neq \hat{0}$. 
\end{defi}

We are now in a position to give a list of equivalent conditions to factorization.

\begin{thm}\label{equivPropFactorSemimod}
Let $L$ be a lattice and let $(A_1, A_2, \dots, A_n)$ be induced by a 
 $\hat{0}$--$\hat{1}$ multichain, $C$.
 Suppose that, for each $y\in L$, if $\bfm{t}\in\mathcal{T}_y^A$, then 
$$
|\supp\bfm{t}|=\rho(y).
$$
Under these conditions the following are equivalent.
\begin{enumerate}
\item For every nonzero $x\in L$, there is an index $i$ such that $|A_i\cap A_x|=1$.
\item For every element $x\in L$ which is the join of two elements from the same $A_j$, there is an index $i$ such that $|A_i\cap A_x|=1$.
\item The multichain $C$ satisfies the meet condition.
\item We have that
\begin{displaymath}
\chi(L,t) = t^{\rho(L)-n} \prod_{i=1}^{n} (t-|A_i|).
\end{displaymath}
\end{enumerate}
\end{thm}
\begin{proof}

$(1)\Rightarrow(4)$ This is Theorem~\ref{bigThm}.

$(4)\Rightarrow (2)$ 
We actually show that $(4)\Rightarrow (1)$ (the fact that $(1)\Rightarrow (2)$ is trivial). We do so by proving the contrapositive. By assumption, there must be a nonzero $x\in L$ such that for each $i$ the number of atoms below $x$ in $A_i$ is different from one.
Let $k$ be the smallest value of $\rho(x)$ for which elements of $L$ have this property. We show that the coefficients of $t^{\rho(L)-k}$ in $\chi(L,t)$ and in $\chi(P,t)=t^{\rho(L)-n}\prod_{i=1}^n (t- |A_i|)$ are different, where $P= \prod_{i=1}^n RT_{\hat{U}(A_i)}$. Using the same proof as we did in Theorem~\ref{bigThm}, we can show that $L\cong P/\sim$. So it suffices to show that the coefficient of $t^{\rho(L)-k}$ in $\chi(P/\sim,t)$ is different from the coefficient in $\chi(P,t)$.

Let $Q$ be the poset obtained by removing all the elements of $P/\sim$ which have rank more than $k$. Let $x_1,x_2,\dots, x_l$ be the elements of $L$ at rank $k$ such that the number of atoms below $x_i$ in each block of the partition is different from one. Then by Lemma~\ref{chainOneNi}, each $x_i$ has atoms above exactly one block. Now let $S=\{\mathcal{T}_{x_1},\mathcal{T}_{x_2}, \dots, \mathcal{T}_{x_l}\}$ be the set of the corresponding transversals. In $Q$, the elements of $S$ are maximal and all the other non-maximal elements in $Q$ satisfy the hypothesis of Lemma~\ref{maxElemQuotient} which can be verified as in the proof of Theorem~\ref{bigThm}. Therefore we can calculate the M\"obius values of the elements of rank $k$ in $Q$ using Lemma~\ref{maxElemQuotient}. Once we know these values we can find the coefficient of $t^{\rho(L)-k}$ in $\chi(P/\sim,t)$.

Each $x_i$ is above at least two atoms and is above only atoms in one block. Therefore the only atomic transversals which are in $ L(\mathcal{T}_{x_i})$ are transversals with single atoms and the transversal with only zeros. Since only atomic transversals have nonzero M\"obius values we get that for all elements of $S$,
$$
c_i\stackrel{\rm def}{=}\sum_{\bfm{t}\in L(\mathcal{T}_{x_i})} \mu(\bfm{t})=1-|A_{x_i}|<0.
$$
We know that $c_i<0$ since the number of atoms below each $x_i$ is at least two. Let $Q_k$ be the set of elements of $Q$ at rank $k$.
Using Lemma~\ref{maxElemQuotient}, we see that the sum of the M\"obius values of $Q_k$ is
\begin{align*}
\sum_{\mathcal{T}_x\in Q_k}\mu(\mathcal{T}_x) 
&=\sum_{i=1}^l \mu(\mathcal{T}_{x_i})+
\sum_{\mathcal{T}_x\in Q_k\setminus S}\mu(\mathcal{T}_x)\\[5pt]
&= \sum_{i=1}^l \left( \sum_{\bfm{t} \in \mathcal{T}_{x_i}} \mu(\bfm{t}) -c_i \right)+ 
\sum_{\mathcal{T}_x\in Q_k\setminus S} \left(\sum_{\bfm{t} \in \mathcal{T}_x} \mu(\bfm{t}) \right).
\end{align*}

As recently noted, only elements of $L$ which have atomic transversals have nonzero M\"obius values. Using this and the assumption that $|\supp\bfm{t}|=\rho(x)=\rho(\mathcal{T}_x)$, we get that the coefficient of  $t^{\rho(L)-k}$ in $\chi(P/\sim,t)$ is
$$
\sum_{|\supp\bfm{t}|=k} \mu(\bfm{t}) - \sum_{i=1}^l c_i
$$
where the first sum is over atomic $\bfm{t}$.
As we saw before, each $c_i$ is negative and all are nonzero and so the coefficient of $t^{\rho(L)-k}$ is different from 
$$
\sum_{|\supp\bfm{t}|=k} \mu(\bfm{t}) 
$$
which is the coefficient of $t^{\rho(L) -k}$ in $\chi(P,t)$. 
This completes the proof that $(4)\Rightarrow (2)$.

$(2) \Rightarrow (3)$ We show the contrapositive holds. Suppose that $C$ does not satisfy the meet condition. Then there is some atomic $x$ which is neither an atom nor $\hat{0}$ such that $x\leq x_i$, $x\not\leq x_{i-1}$, and $x\wedge x_{i-1}=\hat{0}$. It follows that $x$ is only above atoms in $A_i$. Since $x$ is atomic, but not an atom, there are at least two atoms, $a,b$ below $x$ in $A_i$. Let $y=a\vee b$. Since $y\leq x$, $y$ can only be above atoms in $A_i$. Therefore, for all indices $j$, $|A_j \cap A_y|\neq 1$ and $y$ is the join of two atoms.

$(3)\Rightarrow (1)$ First let us note that if $x$ is an atom then the result is obvious. For $x\in L$ let $i$ be the index such that $x\leq x_i$ and $x\not\leq x_{i-1}$. We now induct on $i$. If $i=1$ then it suffices to show that $|A_1|=1$ since then every nonzero $x\le x_1$ is only above the unique element of $A_1$.    However if $a,b$ are distinct atoms in  $A_1$ then $x=a\vee b$ is atomic but not an atom or zero.  Further  $x\le x_1$ but $x\wedge x_{i-1}=x\wedge\hat{0}=\hat{0}$ which contradicts the meet condition.  This finishes the $i=1$ case.

Now suppose that $i>1$ and $x$ is not an atom. Let $z=\bigvee A_x$. Then $z$ is atomic and $A_z=A_x$. Let $y=z\wedge x_{i-1}$. Since $C$ satisfies the meet condition, $y\neq \hat{0}$. By construction $y<x_{i-1}$ and so by induction, there is some index $j\leq i-1$ with $A_j\cap A_y=\{a\}$. Suppose that there was some other atom $b\in A_j \cap A_z$. Then $y\vee b$ is less than or equal to both $z$ and $x_{i-1}$ and so $y\vee b \leq z\wedge x_{i-1}=y$. However, this is impossible since then $A_j\cap A_y\supseteq \{a,b\}$. It follows that $1=|A_j\cap A_z|=|A_j\cap A_x|$ and so (1) holds.
\end{proof}

\begin{figure}
\begin{center}
\begin{tikzpicture}
\node (0) at (1,0) {$\hat{0}$};
\node (a) at (0,1) {$a$}; 
\node (b) at (2,1) {$b$}; 
\node (c) at (0,2) {$c$}; 
\node (d) at (2,2) {$d$}; 
\node (1) at (1,3) {$\hat{1}$}; 
\draw(0)--(a)--(c)--(1);
\draw(0)--(b)--(d)--(1);
\end{tikzpicture}
\caption{A lattice}\label{nonSemiLatticeFig}
\end{center}
\end{figure}
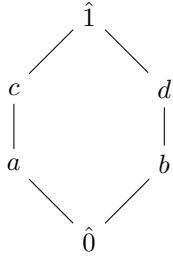

It would be nice if all atomic transversals had the correct support size when using a partition induced by a multichain since then we could remove this assumption from the previous theorem.  Unfortunately this does not always occur. To see why, consider the lattice in Figure~\ref{nonSemiLatticeFig}. The left-most saturated $\hat{0}$--$\hat{1}$ chain
induces the ordered partition
$$
(\{a\},\{b\}).
$$
It is easy to see that the support size of the transversal with both elements is not the rank of their join. Note, however, that if we had the relation $a<d$, then the support size would be the rank of the join. Moreover, note that this would also make the lattice semimodular. We see in the next lemma that semimodularity always implies transversals induced by a multichain have the correct support size.

\begin{lem}\label{chainIndSet} 
Let $L$ be a semimodular lattice and let $(A_1, A_2, \dots, A_n)$ be induced by the multichain $C: \hat{0} =x_0 \leq x_1 \leq  x_2  \leq  \dots\leq x_n=\hat{1}$. If $\sim$ is the standard equivalence, then for all $x\in L$ we have that $\bfm{t}\in \mathcal{T}_x^A$ implies  
$$
|\supp\bfm{t}| = \rho(x).
$$ 
\end{lem}
\begin{proof}
Given an atomic $\bfm{t} \in \mathcal{T}_x^A$ we induct on $|\supp\bfm{t}|$. If $|\supp\bfm{t}|=0$ the result is obvious.

Now suppose that $|\supp\bfm{t}|=k>0$. Let $i$ be the largest index in $\supp\bfm{t}$. Let $\bfm{s} = \bfm{t}(\hat{0}^i)$, then $|\supp\bfm{s}|=k-1$. Suppose that $\bfm{s}\in \mathcal{T}_y^A$ , then $\rho(y)=k-1$ by induction. Let $j$ be the largest index such that $j \in \supp\bfm{s}$. Then $y=\bigvee \bfm{s}\le x_j$ by definition of $j$ and $t_i\not\leq x_j$ since $i>j$. Thus $x=\bigvee\bfm{t}=(\bigvee\bfm{s})\vee t_i>y$. Therefore $\rho(x)>\rho(y)=k-1$ and so $\rho(x)\geq k$. Since $|\supp\bfm{t}|=k$, $\rho(x) \leq k$ as $L$ is semimodular. We conclude that $\rho(x)=k=|\supp\bfm{t}|$ and so our result holds by induction.
\end{proof}

Let us now consider supersolvable semimodular lattices. We begin with a few definitions.  Given a lattice $L$ and $x,z\in L$, we say $(x,z)$ is a \emph{modular pair} if for all $y\leq z$ we have that 
$$
y \vee (x\wedge z) = (y\vee x) \wedge z.
$$
Moreover, we say a multichain $C: x_0=\hat{0}\leq x_1 \leq \dots x_n=\hat{1}$  is \emph{left-modular} if for all $z\in L$  and all $x_i\in C$, every pair $(x_i,z)$ is modular.

Recall that every supersolvable semimodular lattice contains a saturated $\hat{0}$--$\hat{1}$ left-modular chain. It turns out that saturated $\hat{0}$--$\hat{1}$ left-modular chains satisfy the meet condition as we see in the next lemma.
\begin{lem}\label{leftModGoodChain}
Let $L$ be a lattice. If $C: \hat{0} =x_0 \cover x_1 \cover x_2 < \cdots \cover x_n=\hat{1} $ is a left-modular saturated $\hat{0}$--$\hat{1}$ chain then $C$ satisfies the meet condition.
\end{lem}
\begin{proof}
Let $x\in L$ be atomic and neither an atom nor $\hat{0}$. Let $i$ be such that $x\leq x_{i}$ and $x\not\leq x_{i-1}$. Then we claim that there is some atom $a$ with $a<x$ and $a\not\leq x_{i-1}$. To verify the claim, suppose that no such $a$ existed. Since $x$ is not an atom, it must be that all the atoms below $x$ are also below $x_{i-1}$. However, $x$ being atomic implies that $x=\bigvee A_x$ and so $x\leq x_{i-1}$ which is impossible.

By the claim, $x_{i-1}<a\vee x_{i-1} \leq x_i$. Since $x_{i-1} \cover x_i$ we have that $a\vee x_{i-1}=x_i$. Now $(x_{i-1}, x)$ is a modular pair and $a<x$ so, by the definition of a modular pair,
$$
a\vee (x_{i-1} \wedge x) = (a \vee x_{i-1}) \wedge x=x_i\wedge x=x.
$$
But $a<x$ so $x_{i-1} \wedge x\neq \hat{0}$ and thus $C$ satisfies the meet condition.
\end{proof}

We now get Stanley's Supersolvability Theorem as a corollary of Theorem~\ref{equivPropFactorSemimod}, Lemma~\ref{chainIndSet}, and Lemma~\ref{leftModGoodChain}.
\begin{thm}[Stanley's Supersolvability Theorem~\cite{s:sl}]
Let $L$ be a semimodular lattice with partition of the atoms $(A_1, A_2, \dots, A_n)$ induced by a saturated $\hat{0}$--$\hat{1}$ left-modular chain. Then
\begin{displaymath}
\chi(L,t) = \prod_{i=1}^{n} (t-|A_i|).
\end{displaymath}
\end{thm}

 We can use  Theorem~\ref{equivPropFactorSemimod} to give a converse to Stanley's theorem, showing that in a geometric lattice  if one has factorization of $\chi$ using a partition induced by a saturated $\hat{0}$--$\hat{1}$ chain then the chain must   be left-modular.  We just need a couple of definitions for the proof.
Let $L$ be a geometric lattice.  We say a subset, $S$ of the atom set is a \emph{circuit} if $\rho(\vee S) < |S|$ and for all $T\subsetneq S$ we have that $\rho(\vee T) = |T|$. Moreover, we say a partition of the atoms $(A_1,A_2,\dots,A_n)$ satisfies the \emph{circuit condition} if whenever $y,z\in A_j$ there is an $x\in A_i$  with  $i<j$ such that  $\{x,y,z\}$ is a circuit.

\begin{prop}\label{convStanley}
Let $L$ be a geometric lattice with partition of the atoms $(A_1, A_2, \dots, A_n)$ induced by a saturated $\hat{0}$--$\hat{1}$  chain $C: \hat{0} =x_0 \cover x_1 \cover x_2 < \cdots \cover x_n=\hat{1}$. If 
\begin{displaymath}
\chi(L,t) = \prod_{i=1}^{n} (t-|A_i|)
\end{displaymath}
then $C$ is left-modular or, equivalently, $L$ is supersolvable.
\end{prop}
\begin{proof}
 In~\cite[Thm. 2.8]{bz:bccfg}, Bj\"orner and Ziegler show that being supersolvable is equivalent to having a partition of the atom set which satisfies the circuit condition.  So it suffices to show that~(2) of Theorem~\ref{equivPropFactorSemimod} implies that the partition induced by the chain satisfies the circuit condition.  To see why, first note that if condition~(2) of Theorem~\ref{equivPropFactorSemimod}  holds, then whenever $y,z\in A_j$, there is some atom $x\in A_i$, $i\neq j$, with $x<y\vee z$.  It must be the case that $i<j$ since $x\le y\vee z\le x_j$ and the partition was induced by the chain. Moreover, $\{x,y,z\}$ must be a circuit since $x< y\vee z$ so that $\rho(x\vee y\vee z)=\rho(y\vee z)=2$ while sets of atoms of size two or less always have joins whose rank equals the cardinality of the set.

\end{proof}

\section{An Application in Graph Theory}\label{agt}

We will now consider an application of Theorem~\ref{equivPropFactorSemimod} to graph theory. This application was motivated by the computations done in Example~\ref{piExample}. We start with a definition.

\begin{defi}
Let $G$ be a graph with a total ordering of the vertices given by $v_1<v_2<\dots< v_n$. 
Call a subtree of $G$ \emph{increasing} if the vertices along any path starting at its minimum vertex increase in this ordering.
Let $f_k$ be the number of spanning
forests of $G$ with $k$ edges whose components are increasing trees. The \emph{increasing spanning forest generating function} is given by 
\begin{displaymath}
IF(G,t)=\sum_{k=0}^{n-1} (-1)^kf_k t^{n-k}.
\end{displaymath}
\end{defi}

To see what the roots of $IF(G,t)$ are, we will need a partition of the edge set which is given by the ordering on the vertices.

\begin{defi}\label{edgePartDef}
Let $G$ be a graph with a total ordering of the vertices given by $v_1<v_2<\dots< v_n$. Label the edge $v_iv_j$ by $(i,j)$ where $i<j$. The ordered partition $(E_1, E_2, \dots, E_{n-1})$ of the edge set $E(G)$ \emph{induced by the total ordering} is the one with blocks
\begin{displaymath}
E_j = \{(i,j) :\ (i,j)\in E(G)\}.
\end{displaymath}
\end{defi}

It turns out that the sizes of the blocks in the partition are exactly the roots of $IF(G,t)$ as we see in the next theorem.
\begin{thm}\label{incForestThm}
Let $G$ be a graph with the partition $(E_1, E_2,\dots, E_{n})$ induced by the total ordering $v_1<v_2<\dots<v_n$. The increasing spanning forest generating function factors as
$$
IF(G,t)=  \prod_{i=1}^{n} (t-|E_i|).
$$
\end{thm}
\begin{proof} We shall refer to tuples where each element of the tuple is from a different $E_i$ as a transversal (even though there is no underlying poset) and use the term support just as we did previously. We first show that there is a bijection between the set of transversals with support size $k$ for the partition $(E_1,E_2, \dots, E_{n})$ and the set of increasing spanning forests with $k$ edges.

Let $\mathcal{T}_k$ be the set of transversals with support size $k$ and $IF_k$ be the set of increasing spanning forest with $k$ edges. Let $\varphi: \mathcal{T}_k \rightarrow IF_k$ be defined by $\varphi((a_1, b_1), (a_2, b_2), \dots, (a_k,b_k))=F$, where $F$ is the subgraph of $G$ with edges $(a_1, b_1), (a_2, b_2), \dots, (a_k,b_k)$.

It is clear that $F$ has $k$ edges and we claim that $F$ is an increasing spanning forest. The proof that $F$ is acyclic is the same as the one used in Example~\ref{piExample}. Let $T$ be a tree in $F$ which is not increasing. Then in $T$ there must be a vertex $v_m$ which is on a path from the root of $T$ which is preceded and succeeded by vertices of smaller index, $v_a$ and $v_b$. However, this is impossible for the same reasons which force $F$ to be acyclic. It follows that $F$ is an increasing forest and so $\varphi$ is well-defined.

To show $\varphi$ is a bijection, we show it has an inverse. Let $\psi: IF_k \rightarrow \mathcal{T}_k$ be defined by sending the increasing spanning forest to the set of edges it contains. It is obvious that $\varphi$ and $\psi$ are inverses of each other as long as $\psi$ is well-defined.

If $\psi(F)$ is not a transversal, we must have $v_a v_m,v_b v_m\in E(F)$ for some $v_a<v_b<v_m$. Let $T$ be the tree of $F$ containing these edges and let $v_r$ be the minimum vertex of $T$. Since $T$ is increasing and $v_av_m\in E(T)$ with $v_a<v_m$, the unique path from $v_r$ to $v_m$ must contain $v_a$ just prior to $v_m$. By the same token, this path must contain $v_b$ just prior to $v_m$. This is a contradiction, and we conclude that $\psi$ is well-defined.

From above we know that the number of increasing forests with $k$ edges is the same as the number of transversals for the partition $(E_1,E_2,\dots, E_{n})$ with support size $k$. The number of such transversals is $e_k(E_1, E_2, \dots, E_{n})$. Thus we get, 
$$
IF(G,t)= \sum_{k=0}^{n 1} (-1)^kf_k t^{n-k } = \sum_{k=0}^{n} (-1)^ke_k(E_1, E_2, \dots, E_{n })t^{n-k}
$$
from which the result follows.
\end{proof}

Now that we know that the increasing spanning forest generating function always factors, we can use the bond lattice of the graph to show how it relates to the chromatic polynomial. To describe the bond lattice of a graph we need a definition. A \emph{flat}, $F$, of $G$ is a spanning subgraph such that each connected component of $F$ is induced in $G$.   If we then order the flats  of $G$ by inclusion, we obtain the \emph{bond lattice} of $G$.

\begin{thm}\label{incForestThm2}
Denoting the chromatic polynomial of $G$ by $P(G,t)$ we have that
$$
P(G,t)=IF(G,t) 
$$
if and only if $v_1<v_2<\dots< v_n$ is a perfect elimination ordering, i.e., for each $i$, the neighbors of $v_i$ coming before $v_i$ in the ordering form a clique of $G$.
\end{thm}
\begin{proof}
First we note that both $P(G,t)$ and $IF(G,t)$ are multiplicative over connected components of a graph. Additionally, any ordering of the vertices can be restricted to the connected components of a graph and the ordering will be a perfect elimination ordering of the entire graph if and only if its a perfect elimination ordering of each connected component. Therefore it is sufficient to show the result assuming that our graph is connected.

By the previous theorem, $P(G,t)=IF(G,t)$ if and only if 
$$
P(G,t) = \prod_{i=1}^{n} (t-|E_i|)
$$ 
where $(E_1,E_2,\dots,E_{n})$ is induced by the total ordering. This, in turn, is equivalent to
\begin{equation}\label{incPolyEq}
\chi(L,t) = t^{-1}\prod_{i=1}^{n} (t-|E_i|),
\end{equation}
where $L$ is the bond lattice of $G$. Note that we have a $t^{-1}$ on the outside of the previous product.  This is because the rank of  $L$ is $n-1$, whereas we have $n$ blocks in our partition of $E(G)$.  We will still have a polynomial, however, since $E_1=\emptyset$ for any graph.

The partition of the edge set of $G$ gives a partition of the atoms of $L$. Moreover, we claim this partition is induced by a $\hat{0}$--$\hat{1}$ multichain. To verify the claim, we use the multichain $C: \hat{0}=x_0\leq x _1\leq \dots\leq x_{n}=\hat{1}$ where
$$
x_j = \bigvee_{i=1}^j E_{i}.
$$
Note that it is possible that $x_j=x_{j+1}$ since blocks in the partition of the atoms can be empty. 

It is obvious that if $e\in E_k$ then $e\leq x_k$. We must show that $e\not\leq x_{k-1}$. In the graph $x_{k-1}$, all vertices with a label larger than $k$ have degree 0. Since $e=(i, k)$ for some $i$, we have $e\not\leq x_{k-1}$. It follows that $C$ induces the partition $(E_1,E_2, \dots, E_{n})$.

Since the partition is induced by a multichain and $L$ is semimodular, we can apply Theorem~\ref{equivPropFactorSemimod} and Lemma~\ref{chainIndSet}. In particular, using the equivalence of (2) and (4), we have that equation~\ree{incPolyEq} holds
if and only if for any pair $(a,i), (b,i)\in E_{i}$ with $a<b$, there is some index $j$ with a unique atom below $(a,i) \vee (b,i)$ in $E_j$. Since $L$ is the bond lattice of a graph, the only new atom below $(a,i) \vee (b,i)$ is $(a,b)$. It follows that equation~\ree{incPolyEq} holds
if and only if whenever $(a,i), (b,i)\in E(G)$ then $(a,b)\in E(G)$. This is exactly the criteria for $v_1, v_2,\dots, v_n$ to be a perfect elimination ordering of $G$.
\end{proof}

\section{Open Questions and Future Work}

One can weaken the condition of a poset being ranked and still define a characteristic polynomial. In an upcoming article~\cite{h:aqp}, the first author will show that many of the results found in this paper are true when the restriction of being ranked is dropped. Of course, in this case we need a new definition of a ``rank" function. One example, originally defined in~\cite{bs:mfl}, is called \emph{generalized rank}.  We will use it to show that a new family of lattices have characteristic polynomials which factor  with nonnegative integer roots.  In particular, we will show that every interval of a crosscut-simplicial lattice (see~\cite{m:csl} for definition) has such a factorization.  A special case of this result is that every interval of the  m-Tamari lattices (see~\cite{bp:htdhgtp} for definition) has a characteristic polynomial with nice factorization. These results will allow us to recover Blass and Sagan's original result~\cite{bs:mfl}  that the characteristic polynomial of the standard Tamari lattice factors with nonnegative integer roots.  Moreover, we will use a slight modification of Theorem~\ref{bigThm} to show  Blass and Sagan's~\cite{bs:mfl} result  regarding LL lattices.

Additionally, in~\cite{h:aqp}, another use of quotient posets will be demonstrated. Some classic results about the M\"obius function can be proved using induction and quotients. For example, one can use this technique to prove Hall's Theorem~\cite{h:efg}, Rota's Crosscut Theorem~\cite{r:ofct} and Weisner's Theorem~\cite{w:atifs}.

We are investigating whether the methods developed in this paper could be used to show the factorization theorem of Saito~\cite{s:tldflvf} and Terao~\cite{t:gefahstbf} about free arrangements. This would give a combinatorial interpretation of the algebraic property of freeness. 

Another question is whether one can discover a systematic way to find the atom partition for posets which are not lattices. One such example is the weighted partition poset which was introduced in~\cite{dk:cfopcbbho}.

As was shown in Proposition~\ref{convStanley}, the only multichains in a geometric lattice which can satisfy the meet condition are left-modular. This raises the question: what types of multichains can satisfy the meet condition? We know, by Lemma~\ref{leftModGoodChain}, that saturated left-modular chains (in any lattice, not necessarily geometric) do satisfy the meet condition but we do not have any other such families.

It would be very interesting to connect our work with the topology of the order complex of a poset. As a first step, we have been trying to see whether shellability results can be obtained using induction and quotients since this method has already borne fruit as mentioned above. One could also hope to find connections with discrete Morse theory using these ideas.

Finally, we gave a definition of an increasing spanning forest and showed that its generating function always factors. This raises the question of whether our theorem is a special case of a more general result about the Tutte polynomial of a matroid. Of course, we would first need a definition of what it means for an independent set of a matroid to be increasing.  
 In joint work with  Martin, we have succeeded in generalizing the results of the previous section to arbitrary simplicial complexes.

\section{Acknowledgments}
The authors would like to thank  Torsten Hoge  for pointing out the connection of Theorem~\ref{bigThmAtomicVer} with Terao's theorem  about   nice partitions in central hyperplane arrangements.  They would also like to thank the two anonymous referees for their helpful suggestions.

 \newcommand{\noop}[1]{} \def\cprime{$'$}

\end{document}